\documentclass[11pt,twoside]{article}
\usepackage{latexsym,amsmath,amsthm,amssymb,amsfonts}
\usepackage{mathrsfs}
\usepackage{amscd}
\usepackage{cite}

\usepackage[ansinew]{inputenc}
\numberwithin{equation}{section}

\newtheorem{theorem}{\bf Theorem}[section]

\newtheorem{remark}[theorem]{\bf Remark}
\newtheorem{lemma}[theorem]{\bf Lemma}

\newtheorem{corollary}[theorem]{\bf Corollary}
\catcode`@=11

\newskip\plaincentering \plaincentering=0pt plus 1000pt minus 1000pt
\def\@plainlign{\tabskip=0pt\everycr={}}
\def\eqalignno#1{\displ@y \tabskip\plaincentering
  \halign to\displaywidth{\hfil$\@lign\displaystyle{##}$\tabskip\z@skip
    &$\@lign\displaystyle{{}##}$\hfil\tabskip\plaincentering
    &\llap{$\@lign##$}\tabskip\z@skip\crcr
    #1\crcr}}
\def\leqalignno#1{\displ@y \tabskip\plaincentering
  \halign to\displaywidth{\hfil$\@lign\displaystyle{##}$\tabskip\z@skip
    &$\@lign\displaystyle{{}##}$\hfil\tabskip\plaincentering
    &\kern-\displaywidth\rlap{$\@lign##$}\tabskip\displaywidth\crcr
    #1\crcr}}
\def\plainLet@{\relax\iffalse{\fi\let\\=\cr\iffalse}\fi}
\def\plainvspace@{\def\vspace##1{\noalign{\vskip##1}}}

\def\intic@{\mathchoice{\hskip5\p@}{\hskip4\p@}{\hskip4\p@}{\hskip4\p@}}
\def\negintic@
 {\mathchoice{\hskip-5\p@}{\hskip-4\p@}{\hskip-4\p@}{\hskip-4\p@}}
\def\intkern@{\mathchoice{\!\!\!}{\!\!}{\!\!}{\!\!}}
\def\intdots@{\mathchoice{\cdots}{{\cdotp}\mkern1.5mu
    {\cdotp}\mkern1.5mu{\cdotp}}{{\cdotp}\mkern1mu{\cdotp}\mkern1mu
      {\cdotp}}{{\cdotp}\mkern1mu{\cdotp}\mkern1mu{\cdotp}}}
\newcount\intno@
\def\iint{\intno@=\tw@\futurelet\next\ints@}
\def\iiint{\intno@=\thr@@\futurelet\next\ints@}
\def\iiiint{\intno@=4 \futurelet\next\ints@}
\def\idotsint{\intno@=\z@\futurelet\next\ints@}
\def\ints@{\findlimits@\ints@@}
\newif\iflimtoken@
\newif\iflimits@
\def\findlimits@{\limtoken@false\limits@false\ifx\next\limits
 \limtoken@true\limits@true\else\ifx\next\nolimits\limtoken@true\limits@false
    \fi\fi}
\def\multintlimits@{\intop\ifnum\intno@=\z@\intdots@
  \else\intkern@\fi
    \ifnum\intno@>\tw@\intop\intkern@\fi
     \ifnum\intno@>\thr@@\intop\intkern@\fi\intop}
\def\multint@{\int\ifnum\intno@=\z@\intdots@\else\intkern@\fi
   \ifnum\intno@>\tw@\int\intkern@\fi
    \ifnum\intno@>\thr@@\int\intkern@\fi\int}
\def\ints@@{\iflimtoken@\def\ints@@@{\iflimits@
   \negintic@\mathop{\intic@\multintlimits@}\limits\else
    \multint@\nolimits\fi\eat@}\else
     \def\ints@@@{\multint@\nolimits}\fi\ints@@@}
\def\Sb{_\bgroup\vspace@
        \baselineskip=\fontdimen10 \scriptfont\tw@
        \advance\baselineskip by \fontdimen12 \scriptfont\tw@
        \lineskip=\thr@@\fontdimen8 \scriptfont\thr@@
        \lineskiplimit=\thr@@\fontdimen8 \scriptfont\thr@@
        \Let@\vbox\bgroup\halign\bgroup \hfil$\scriptstyle
            {##}$\hfil\cr}
\def\endSb{\crcr\egroup\egroup\egroup}
\def\Sp{^\bgroup\vspace@
        \baselineskip=\fontdimen10 \scriptfont\tw@
        \advance\baselineskip by \fontdimen12 \scriptfont\tw@
        \lineskip=\thr@@\fontdimen8 \scriptfont\thr@@
        \lineskiplimit=\thr@@\fontdimen8 \scriptfont\thr@@
        \Let@\vbox\bgroup\halign\bgroup \hfil$\scriptstyle
            {##}$\hfil\cr}
\def\endSp{\crcr\egroup\egroup\egroup}
\def\Let@{\relax\iffalse{\fi\let\\=\cr\iffalse}\fi}
\def\vspace@{\def\vspace##1{\noalign{\vskip##1 }}}
\def\aligned{\,\vcenter\bgroup\plainvspace@\plainLet@\openup\jot\m@th\ialign
  \bgroup \strut\hfil$\displaystyle{##}$&$\displaystyle{{}##}$\hfil\crcr}
\def\endaligned{\crcr\egroup\egroup}
\def\matrix{\,\vcenter\bgroup\plainLet@\plainvspace@
    \normalbaselines
  \m@th\ialign\bgroup\hfil$##$\hfil&&\quad\hfil$##$\hfil\crcr
    \mathstrut\crcr\noalign{\kern-\baselineskip}}
\def\endmatrix{\crcr\mathstrut\crcr\noalign{\kern-\baselineskip}\egroup
                \egroup\,}
\newtoks\hashtoks@
\hashtoks@={#}
\def\format{\crcr\egroup\iffalse{\fi\ifnum`}=0 \fi\format@}
\def\format@#1\\{\def\preamble@{#1}%
  \def\c{\hfil$\the\hashtoks@$\hfil}%
  \def\r{\hfil$\the\hashtoks@$}%
  \def\l{$\the\hashtoks@$\hfil}%
  \setbox\z@=\hbox{\xdef\Preamble@{\preamble@}}\ifnum`{=0 \fi\iffalse}\fi
   \ialign\bgroup\span\Preamble@\crcr}

\def\cases{\left\{\,\vcenter\bgroup\plainvspace@
     \normalbaselines\openup\jot\m@th
      \plainLet@\ialign\bgroup$\displaystyle{##}$\hfil&\quad$\displaystyle{{}##}$\hfil\crcr
      \mathstrut\crcr\noalign{\kern-\baselineskip}}

\newif\iftagsleft@
\tagsleft@true
\def\TagsOnRight{\global\tagsleft@false}
\def\tag#1$${\iftagsleft@\leqno\else\eqno\fi
 \hbox{\def\pagebreak{\global\postdisplaypenalty-\@M}%
 \def\nopagebreak{\global\postdisplaypenalty\@M}\rm(#1\unskip)}%
  $$\postdisplaypenalty\z@\ignorespaces}
\interdisplaylinepenalty=\@M
\def\plainallowdisplaybreak@{\def\allowdisplaybreak{\noalign{\allowbreak}}}
\def\plaindisplaybreak@{\def\displaybreak{\noalign{\break}}}
\def\align#1\endalign{\def\tag{&}\plainvspace@\plainallowdisplaybreak@\plaindisplaybreak@
  \iftagsleft@\plainlalign@#1\endalign\else
   \plainralign@#1\endalign\fi}
\def\plainralign@#1\endalign{\displ@y\plainLet@\tabskip\plaincentering\halign to\displaywidth
     {\hfil$\displaystyle{##}$\tabskip=\z@&$\displaystyle{{}##}$\hfil
       \tabskip=\plaincentering&\llap{\hbox{\rm(##\unskip)}}\tabskip\z@\crcr
             #1\crcr}}
\def\plainlalign@
 #1\endalign{\displ@y\plainLet@\tabskip\plaincentering\halign to \displaywidth
   {\hfil$\displaystyle{##}$\tabskip=\z@&$\displaystyle{{}##}$\hfil
   \tabskip=\plaincentering&\kern-\displaywidth
        \rlap{\hbox{\rm(##\unskip)}}\tabskip=\displaywidth\crcr
               #1\crcr}}

\def\re@#1{\par\hangindent\parindent\indent\llap{#1\enspace}\ignorespaces}
\def\qfootnote#1{\edef\@sf{\spacefactor\the\spacefactor}{}#1\@sf
      \insert\footins{\let\egroup=}\footnotesize 
      \interlinepenalty100 \let\par=\endgraf
        \leftskip=0pt \rightskip=0pt
        \splittopskip=10pt plus 1pt minus 1pt \floatingpenalty=20000
   \smallskip\re@{#1}\bgroup\strut\aftergroup{\strut\egroup}\let\next}
\catcode`\@=\active
\TagsOnRight

\begin{document}
\title{\bf Universal bounds on eigenvalues of the buckling problem on spherical domains
\footnote{This research is supported by NSFC of China (No.
10671181), Project of Henan Provincial department of Sciences and
Technology (No. 092300410143), and NSF of Henan Provincial Education
department (No. 2009A110010).}}
\author{Guangyue Huang,\ Xingxiao Li
\footnote{The corresponding author. Email: xxl$@$henannu.edu.cn}\, , \  Linfen Cao\\
{\normalsize Department of Mathematics, Henan Normal University}
\\{\normalsize Xinxiang 453007, Henan, P.R. China}}
\date{}
\maketitle
\begin{quotation}
\noindent{\bf Abstract.} In this paper we study the eigenvalues of
buckling problem on domains in a unit sphere. By introducing a new
parameter and using Cauchy inequality, we optimize the
inequality obtained by Wang and Xia in \cite{wangxia07}.

\end{quotation}

\section{Introduction}

Let $\Omega$ be a connected bounded domain in an
$n(\geq2)$-dimensional Euclidean space $\mathbb{R}^n$ and $\nu$ be
the unit outward normal vector field of $\partial\Omega$. The
well-known eigenvalue problem
\begin{equation}\label{Intro1}\left\{\begin{array}{ll}
\Delta^2 u=\Lambda(-\Delta) u & {\rm in}\ \Omega,\\
u=\frac{\partial u}{\partial\nu}=0 & {\rm on}\
\partial\Omega
\end{array}\right.
\end{equation}
is called a {\em buckling problem}, which is used to describe the
critical buckling load of a clamped plate subjected to a uniform
compressive force around its boundary.

Let $$0<\Lambda_1\leq\Lambda_2\leq \Lambda_3\leq \cdots$$ denote the
successive eigenvalues for \eqref{Intro1}, where each eigenvalue is
repeated according to its multiplicity. In 1956,
Payne-P\'{o}lya-Weinberger \cite{ppw2} proved
\begin{equation}\label{Intro2}
\Lambda_2\leq 3\Lambda_1\ \ \ {\rm for}\ \Omega\in \mathbb{R}^2.
\end{equation} Following the method of Payne-P\'{o}lya-Weinberger in \cite{ppw2},
it reads that for $\Omega\in\mathbb{R}^n$ (for the generalization of
\eqref{Intro2} to $n$ dimensions, see \cite{ash04}):
$$\Lambda_2\leq \left(1+\frac8{n+2}\right)\Lambda_1.
$$
Subsequently, in 1984, Hile and Yeh \cite{hile1} improved the above
inequality as follows:
$$
\Lambda_2\leq \frac{n^2+8n+20}{(n+2)^2}\Lambda_1.
$$
In 1998, Ashbaugh \cite{ash} obtained
$$\sum_{i=1}^n\Lambda_{i+1}\leq (n+4)\Lambda_1.$$
In a recent survey paper, answering a question of Ashbaugh in
\cite{ash}, Cheng-Yang \cite{chengyang06} proved the following
universal inequalities on eigenvalues for the eigenvalue problem
\eqref{Intro1}:
$$
\sum_{i=1}^k(\Lambda_{k+1}-\Lambda_i)^2\leq\frac{4(n+2)}{
n^2}\sum_{i=1}^k(\Lambda_{k+1}-\Lambda_i)\Lambda_i. $$

It has become clear that many similar inequalities on eigenvalues of
Laplacian of Payne-P\'{o}lya-Weinberger rely on facts involving
operators, their commutators, and traces. For the related research
and improvement in this direction, see
\cite{harrell97,huang09,levitin,huangchen,harrell,wang07,cheng4} and
the references therein.

Let $x_1,x_2,\ldots,x_{n+1}$ be the standard Euclidean coordinate
functions of $\mathbb{R}^{n+1}$. Then the unit sphere is defined by
$$\mathbb{S}^n=\left\{ (x_1,x_2,\ldots, x_{n+1})\in\mathbb{R}^{n+1}\ ;\
\sum_{\alpha=1}^{n+1}x_\alpha^2=1\right\}.$$ In 2007, Wang and Xia
\cite{wangxia07} considered the buckling problem on domains in a
unit sphere $\mathbb{S}^n$ and obtained the following result:

\begin{theorem}{\rm\cite{wangxia07}} Let $\Lambda_i$ be the
$i^{th}$ eigenvalue of the following eigenvalue problem:
$$\Delta^2 u=\Lambda(-\Delta) u \ \ {\rm in}\ \Omega,\ \ \ \
u=\frac{\partial u}{\partial\nu}=0\ \ {\rm on}\
\partial\Omega,$$ where $\Omega$ is a connected domain in
a unit sphere $\mathbb{S}^n(n\geq2)$ with smooth boundary
$\partial\Omega$ and $\nu$ is the unit outward normal vector field
of $\partial\Omega$. Then for any $\delta>0$,
\begin{equation}\label{Intro3}
\aligned 2\sum_{i=1}^k(\Lambda_{k+1}-\Lambda_i)^2\leq&
\sum_{i=1}^k(\Lambda_{k+1}-\Lambda_i)^2\left(\delta\Lambda_i
+\frac{\delta^2(\Lambda_i-(n-2))}{4(\delta\Lambda_i+n-2)}\right)\\
&+
\frac{1}{\delta}\sum_{i=1}^k(\Lambda_{k+1}-\Lambda_i)\left(\Lambda_i
+\frac{(n-2)^2}{4}\right).
\endaligned
\end{equation}
\end{theorem}

We remark that the right hand side of inequality \eqref{Intro3}
depends on $\delta$. In the current paper, by introducing a new
parameter and using Cauchy inequality, we obtain a stronger
inequality than \eqref{Intro3} which is independent of $\delta$, and
derive an inequality of the type of Yang (see inequality
\eqref{Intro5}). Our main results are stated as follows:

\begin{theorem}\label{th:main} Let $\Lambda_i$ be the $i^{th}$ eigenvalue
of the following eigenvalue problem:
$$\Delta^2 u=\Lambda(-\Delta) u \ \ {\rm in}\ \Omega,\ \ \ \
u=\frac{\partial u}{\partial\nu}=0\ \ {\rm on}\
\partial\Omega,$$ where $\Omega$ is a connected domain in
a unit sphere $\mathbb{S}^n(n\geq2)$ with smooth boundary
$\partial\Omega$ and $\nu$ is the unit outward normal vector field
of $\partial\Omega$. Then
\begin{alignat}{1}\label{Intro4}
&\sum_{i=1}^k(\Lambda_{k+1}-\Lambda_i)^2
\left(2+\frac{n-2}{\Lambda_i-(n-2)}\right)\\
\leq&
2\left\{\sum_{i=1}^k(\Lambda_{k+1}-\Lambda_i)^2\left(\Lambda_i
-\frac{n-2}{\Lambda_i-(n-2)}\right)\right\}^{1/2}\\
&\qquad \qquad\qquad
\times\left\{\sum_{i=1}^k(\Lambda_{k+1}-\Lambda_i)
\left(\Lambda_i+\frac{(n-2)^2}{4}\right)\right\}^{1/2}.
\end{alignat}
\end{theorem}

\begin{corollary} Under the assumptions of Theorem 1.2,
\begin{equation}\label{Intro5}
\aligned\sum_{i=1}^k(\Lambda_{k+1}-\Lambda_i)^2\leq
\sum_{i=1}^k(\Lambda_{k+1}-\Lambda_i)\left(\Lambda_i
-\frac{n-2}{\Lambda_i-(n-2)}\right)
\left(\Lambda_i+\frac{(n-2)^2}{4}\right),
\endaligned
\end{equation}
\begin{equation}\label{Intro6}\Lambda_{k+1}\leq
S_{k+1}+\sqrt{S_{k+1}^2-T_{k+1}},
\end{equation}
\begin{equation}\label{Intro7}\Lambda_{k+1}-\Lambda_k\leq2\sqrt{S_{k+1}^2-T_{k+1}},
\end{equation}
where \begin{equation}\label{Intro8}S_{k+1}=\frac1{
k}\sum_{i=1}^k\Lambda_i+\frac{1}{
2k}\sum_{i=1}^k\left(\Lambda_i-\frac{n-2}{\Lambda_i-(n-2)}\right)
\left(\Lambda_i+\frac{(n-2)^2}{4}\right),\end{equation}
\begin{equation}\label{Intro9}T_{k+1}=\frac{1}{k}\sum_{i=1}^k\Lambda_i^2+\frac{1}{
k}\sum_{i=1}^k\Lambda_i\left(\Lambda_i-\frac{n-2}{\Lambda_i-(n-2)}\right)
\left(\Lambda_i+\frac{(n-2)^2}{4}\right).\end{equation}\end{corollary}

\begin{remark} {\rm The inequality \eqref{Intro4} is equivalent to
$$\aligned 2\sum_{i=1}^k(\Lambda_{k+1}-\Lambda_i)^2
\leq&-\sum_{i=1}^k(\Lambda_{k+1}-\Lambda_i)^2\frac{n-2}{\Lambda_i-(n-2)}\\
&+
2\left\{\sum_{i=1}^k(\Lambda_{k+1}-\Lambda_i)^2
\left(\Lambda_i-\frac{n-2}{\Lambda_i-(n-2)}\right)\right\}^{1/2}\\
&\times\left\{\sum_{i=1}^k(\Lambda_{k+1}-\Lambda_i)
\left(\Lambda_i+\frac{(n-2)^2}{4}\right)\right\}^{1/2}.\endaligned$$
From the inequality \eqref{add} in Section 2 and making use of
Cauchy inequality, the following inequality can be deduced:
$$\aligned
-\sum_{i=1}^k&(\Lambda_{k+1}-\Lambda_i)^2\frac{n-2}{\Lambda_i-(n-2)}\\
&+
2\left\{\sum_{i=1}^k(\Lambda_{k+1}-\Lambda_i)^2\left(\Lambda_i
-\frac{n-2}{\Lambda_i-(n-2)}\right)\right\}^{1/2}\\
&\times\left\{\sum_{i=1}^k(\Lambda_{k+1}-\Lambda_i)
\left(\Lambda_i+\frac{(n-2)^2}{4}\right)\right\}^{1/2}\\
\leq&-\sum_{i=1}^k(\Lambda_{k+1}-\Lambda_i)^2\frac{n-2}{\Lambda_i-(n-2)}\\
&+
\delta\sum_{i=1}^k(\Lambda_{k+1}-\Lambda_i)^2\left(\Lambda_i
-\frac{n-2}{\Lambda_i-(n-2)}\right)\\
&+\frac{1}{\delta}\sum_{i=1}^k(\Lambda_{k+1}-\Lambda_i)
\left(\Lambda_i+\frac{(n-2)^2}{4}\right)\\
=&\sum_{i=1}^k(\Lambda_{k+1}-\Lambda_i)^2\left(\delta\Lambda_i
-\frac{(\delta+1)(n-2)}{\Lambda_i-(n-2)}\right)\\
&+
\frac{1}{\delta}\sum_{i=1}^k(\Lambda_{k+1}-\Lambda_i)
\left(\Lambda_i+\frac{(n-2)^2}{4}\right)\\
\leq&\sum_{i=1}^k(\Lambda_{k+1}-\Lambda_i)^2\left(\delta\Lambda_i
+\frac{\delta^2(\Lambda_i-(n-2))}{4(\delta\Lambda_i+n-2)}\right)\\
&+ \frac{1}{\delta}\sum_{i=1}^k(\Lambda_{k+1}
-\Lambda_i)\left(\Lambda_i+\frac{(n-2)^2}{4}\right),
\endaligned$$
which shows that the inequality \eqref{Intro4} is sharper than
inequality \eqref{Intro3}. Therefore, Theorem 1.2 improves Theorem
1.1.}\end{remark}

\section{Proof of Theorem}

By the method of constructing trial functions introduced by Cheng
and Yang in \cite{chengyang06}, for any $\alpha=1,\ldots, n+1$ and
each $i=1,\ldots,k$, the vector-valued functions $x_\alpha\nabla
u_i$ can be decomposed as
\begin{equation}\label{2sec1}x_\alpha\nabla u_i=\nabla h_{\alpha i}+W_{\alpha i},\end{equation}
where $h_{\alpha i}\in H^2_{2,D}(\Omega)$, $\nabla h_{\alpha i}$ is
the projection of $x_\alpha\nabla u_i$ in
$\mathbf{H}^2_{1,D}(\Omega)$ and $W_{\alpha i}\perp
\mathbf{H}^2_{1,D}(\Omega)$ (for the definitions of
$H^2_{2,D}(\Omega)$ and $\mathbf{H}^2_{1,D}(\Omega)$, we refer to
\cite{wangxia07}). Hence,
$$W_{\alpha i}|_{\partial\Omega}=0$$ and
$$ \int_\Omega\langle W_{\alpha i},\nabla u\rangle=0, \ \ \ \ {\rm for\ any}\ u\in  H^2_{2,D}(\Omega).$$

Define $\|f\|^2=\int_\Omega|f|^2$. Then (see inequalities (2.19) and
(2.40) in \cite{wangxia07})
\begin{equation}\label{2sec2}
\|x_\alpha\nabla u_i\|^2=\|\nabla h_{\alpha i}\|^2+\|W_{\alpha
i}\|^2
\end{equation}
and for any $\delta>0$,
\begin{alignat}{1}\label{2sec3}
\sum_{i=1}^k(\Lambda_{k+1}-\Lambda_i)^2r_{\alpha
i}\leq&\sum_{i=1}^k(\Lambda_{k+1}-\Lambda_i)^2\Big(\delta p_{\alpha
i}+(\delta\Lambda_i+n-2)\|W_{\alpha i}\|^2\\
&+\delta\|\langle\nabla x_\alpha,\nabla u_i\rangle\|^2 \Big)
+\frac{1}{\delta}\sum_{i=1}^k(\Lambda_{k+1}-\Lambda_i)\|Z_{\alpha
i}\|^2,\end{alignat}
where
$$\aligned
r_{\alpha i}=&2\|\langle\nabla x_\alpha,\nabla
u_i\rangle\|^2+\int_\Omega \langle\nabla x_\alpha^2,\Delta u_i\nabla
u_i\rangle+(n-2)\| x_\alpha\nabla u_i\|^2,\\p_{\alpha
i}=&\int_\Omega\langle\nabla x_\alpha^2,u_i\nabla (\Delta
u_i)+\Lambda_iu_i\nabla u_i\rangle,\\
Z_{\alpha i}=&\nabla\langle\nabla x_\alpha,\nabla
u_i\rangle-\frac{n-2}{ 2}x_\alpha\nabla u_i.\endaligned
$$

Next, we are to optimize the inequality (2.41) in \cite{wangxia07}.
Let $C$ be a positive constant. Then it follows from Cauchy
inequality that
\begin{alignat}{1}\label{2sec4} C\|\langle\nabla x_\alpha,\nabla
u_i\rangle\|^2=&C\int_\Omega\langle\nabla x_\alpha,\nabla
u_i\rangle^2\\
=&-C\int_\Omega  x_\alpha {\rm div}(\langle\nabla
x_\alpha,\nabla u_i\rangle\nabla u_i)\\
=&-C\int_\Omega\langle x_\alpha\nabla u_i,\nabla \langle\nabla
x_\alpha,\nabla u_i\rangle\rangle-C\int_\Omega\langle\nabla
x_\alpha,\nabla u_i\rangle x_\alpha\Delta u_i\\
=&-C\int_\Omega\langle \nabla h_{\alpha i},\nabla \langle\nabla
x_\alpha,\nabla
u_i\rangle\rangle-\frac{C}{2}\int_\Omega\langle\nabla
x_\alpha^2,\nabla u_i\rangle \Delta u_i\\
\leq&(\delta\Lambda_i+n-2)\|\nabla h_{\alpha
i}\|^2+\frac{C^2}{4(\delta\Lambda_i+n-2)}\|\nabla \langle\nabla
x_\alpha,\nabla u_i\rangle\|^2\\
&-\frac{C}{2}\int_\Omega\langle\nabla x_\alpha^2,\nabla u_i\rangle
\Delta u_i,
\end{alignat}
where ${\rm div}(Z)$ denotes the divergence of $Z$. Applying
\eqref{2sec4} to \eqref{2sec3} yields
\begin{alignat}{1}\label{2sec5}
\sum_{i=1}^k(\Lambda_{k+1}-\Lambda_i)^2r_{\alpha
i}\leq&\sum_{i=1}^k(\Lambda_{k+1}-\Lambda_i)^2\Big(\delta p_{\alpha
i}+(\delta\Lambda_i+n-2)\|W_{\alpha
i}\|^2\\&+(\delta-C)\|\langle\nabla x_\alpha,\nabla
u_i\rangle\|^2+C\|\langle\nabla x_\alpha,\nabla
u_i\rangle\|^2\Big)\\
&+\frac{1}{\delta}\sum_{i=1}^k(\Lambda_{k+1}-\Lambda_i)\|Z_{\alpha
i}\|^2\\
\leq&\sum_{i=1}^k(\Lambda_{k+1}-\Lambda_i)^2\Big(\delta p_{\alpha
i}+(\delta\Lambda_i+n-2)(\|W_{\alpha i}\|^2+\|\nabla h_{\alpha
i}\|^2)\\&+(\delta-C)\|\langle\nabla x_\alpha,\nabla
u_i\rangle\|^2+\frac{C^2}{4(\delta\Lambda_i+n-2)}\|\nabla
\langle\nabla x_\alpha,\nabla
u_i\rangle\|^2\\
&-\frac{C}{2}\int_\Omega\langle\nabla x_\alpha^2,\nabla u_i\rangle
\Delta
u_i\Big)+\frac{1}{\delta}\sum_{i=1}^k(\Lambda_{k+1}-\Lambda_i)\|Z_{\alpha
i}\|^2\\
=&\sum_{i=1}^k(\Lambda_{k+1}-\Lambda_i)^2\Big(\delta p_{\alpha
i}+(\delta\Lambda_i+n-2)\|x_\alpha\nabla
u_i\|^2\\&+(\delta-C)\|\langle\nabla x_\alpha,\nabla
u_i\rangle\|^2+\frac{C^2}{4(\delta\Lambda_i+n-2)}\|\nabla
\langle\nabla x_\alpha,\nabla
u_i\rangle\|^2\\
&-\frac{C}{2}\int_\Omega\langle\nabla x_\alpha^2,\nabla u_i\rangle
\Delta
u_i\Big)+\frac{1}{\delta}\sum_{i=1}^k(\Lambda_{k+1}-\Lambda_i)\|Z_{\alpha
i}\|^2,
\end{alignat}
where in the last step in \eqref{2sec5} we have used equality
\eqref{2sec2}. A direct calculation yields (see (2.44), (2.45),
(2.46) and (2.47) in \cite{wangxia07})
\begin{equation}\label{2sec6}
\sum_{\alpha=1}^{n+1}r_{\alpha i}=n,
\end{equation}
\begin{equation}\label{2sec7}
\sum_{\alpha=1}^{n+1}p_{\alpha i}=0,
\end{equation}
\begin{equation}\label{2sec8}
\sum_{\alpha=1}^{n+1}\|x_\alpha\nabla
u_i\|^2=\sum_{\alpha=1}^{n+1}\|\langle\nabla x_\alpha,\nabla
u_i\rangle\|^2=1,
\end{equation}
\begin{equation}\label{2sec9}
\sum_{\alpha=1}^{n+1}\|\nabla \langle\nabla x_\alpha,\nabla
u_i\rangle\|^2=\Lambda_i-(n-2)
\end{equation}
and
\begin{equation}\label{2sec10}
\sum_{\alpha=1}^{n+1}\|Z_{\alpha i}\|^2=\Lambda_i+\frac{(n-2)^2}{
4}.
\end{equation}
Therefore, summing up \eqref{2sec5} over $\alpha$ from 1 to $n+1$,
one gets
$$\aligned
n\sum_{i=1}^k(\Lambda_{k+1}-\Lambda_i)^2\leq&\sum_{i=1}^k(\Lambda_{k+1}-\Lambda_i)^2
\Big(\delta\Lambda_i+n-2+(\delta-C)\\
&+\frac{C^2}{4(\delta\Lambda_i+n-2)}(\Lambda_i-(n-2))\Big)\\
&+\frac1\delta\sum_{i=1}^k(\Lambda_{k+1}-\Lambda_i)\left(\Lambda_i+\frac{(n-2)^2}{
4}\right).
\endaligned$$
That is,
\begin{alignat}{1}\label{2sec11}
2\sum_{i=1}^k(\Lambda_{k+1}-\Lambda_i)^2\leq&\sum_{i=1}^k(\Lambda_{k+1}-\Lambda_i)^2
\left(\delta\Lambda_i+(\delta-C)\right.\\
&\left.+\frac{C^2}{
4(\delta\Lambda_i+n-2)}(\Lambda_i-(n-2))\right)\\
&
+\frac{1}{\delta}\sum_{i=1}^k(\Lambda_{k+1}-\Lambda_i)\left(\Lambda_i+\frac{(n-2)^2}{
4}\right).
\end{alignat}
Taking $$C=\frac{2(\delta\Lambda_i+n-2)}{\Lambda_i-(n-2)}$$ in
\eqref{2sec11} yields
\begin{alignat}{1}\label{2sec12}
2\sum_{i=1}^k(\Lambda_{k+1}-\Lambda_i)^2\leq&\sum_{i=1}^k(\Lambda_{k+1}-\Lambda_i)^2
\left(\delta\Lambda_i-\frac{(\delta+1)(n-2)}{\Lambda_i-(n-2)}\right)\\
&+\frac{1}{\delta}\sum_{i=1}^k(\Lambda_{k+1}-\Lambda_i)\left(\Lambda_i+\frac{(n-2)^2}{
4}\right),
\end{alignat}
and hence
\begin{alignat}{1}\label{2sec13}
&\sum_{i=1}^k(\Lambda_{k+1}-\Lambda_i)^2
\left(2+\frac{n-2}{\Lambda_i-(n-2)}\right)\\
\leq& \delta\sum_{i=1}^k(\Lambda_{k+1}-\Lambda_i)^2\left(\Lambda_i
-\frac{n-2}{\Lambda_i-(n-2)}\right)\\
&+\frac{1}{\delta}\sum_{i=1}^k(\Lambda_{k+1}-\Lambda_i)
\left(\Lambda_i+\frac{(n-2)^2}{ 4}\right).
\end{alignat}

To complete the proof of Theorem \ref{th:main}, we need the
following lemma:

\begin{lemma} Let $\Omega$ be a connected bounded domain in
$\mathbb{S}^n$. Then
\begin{equation}\label{2sec14}\Lambda_1\geq
n.\end{equation}\end{lemma}

\begin{proof} Let $\Omega_1$, $\Omega_2$ be two connected bounded
domains in $\mathbb{S}^n$ and $\Omega_1\subset\Omega_2$. Let
$u_1(\Omega_1)$ be the eigenfunction corresponding to
$\Lambda_1(\Omega_1)$. Then the function defined by $$\widetilde{
u_1}=\left\{\begin{array}{ll} u_1 &\ \ \  {\rm in }\ \ \Omega_1,\\
0 &\ \ \  {\rm in }\ \ \Omega_2- \Omega_1
\end{array}\right.
$$
is a eigenfunction corresponding to $\Lambda_1(\Omega_2)$. Hence,
$\Lambda_1(\Omega_1)\geq \Lambda_1(\Omega_2)$. Denote by $\lambda_1$
the first eigenvalue of Laplacian. It is easy to see
$\Lambda_1(\mathbb{S}^n)=\lambda_1(\mathbb{S}^n)=n$ because there
are no boundary conditions in this case. It follows that
$\Lambda_1(\Omega)\geq \Lambda_1(\mathbb{S}^n)=n$ by setting
$\Omega=\Omega_1$ and $\Omega_2=\mathbb{S}^n$. This completes the
proof of Lemma 2.1.\end{proof}

Inequality \eqref{2sec14} shows that for any $i$,
\begin{equation}\label{add}\Lambda_i-\frac{n-2}{\Lambda_i-(n-2)}>0.\end{equation}
Minimizing the right hand side of \eqref{2sec13} as a function of
$\delta$ by choosing
$$\delta=\left(\frac{\sum_{i=1}^k(\Lambda_{k+1}-\Lambda_i)
\left(\Lambda_i+\frac{(n-2)^2}{
4}\right)}{\sum_{i=1}^k(\Lambda_{k+1}-\Lambda_i)^2
\left(\Lambda_i-\frac{n-2}{\Lambda_i-(n-2)}\right)}\right)^{1/2},$$
we obtain the inequality \eqref{Intro4}, completing the proof of
Theorem \ref{th:main}.

\vspace*{3mm}

{\bf Proof of Corollary 1.3.} It is easy to see from \eqref{Intro4}
that
\begin{alignat}{1}\label{2sec15}
\sum_{i=1}^k(\Lambda_{k+1}-\Lambda_i)^2\leq&
\left\{\sum_{i=1}^k(\Lambda_{k+1}-\Lambda_i)^2
\left(\Lambda_i-\frac{n-2}{\Lambda_i-(n-2)}\right)\right\}^{1/2}\\
&\qquad\qquad\times\left\{\sum_{i=1}^k(\Lambda_{k+1}-\Lambda_i)
\left(\Lambda_i+\frac{(n-2)^2}{4}\right)\right\}^{1/2}.
\end{alignat}
One can check by induction that
$$\aligned
&\left\{\sum_{i=1}^k(\Lambda_{k+1}-\Lambda_i)^2
\left(\Lambda_i-\frac{n-2}{\Lambda_i-(n-2)}\right)\right\}
\left\{\sum_{i=1}^k(\Lambda_{k+1}-\Lambda_i)
\left(\Lambda_i+\frac{(n-2)^2}{
4}\right)\right\}\\
\leq&\left\{\sum_{i=1}^k(\Lambda_{k+1}-\Lambda_i)^2\right\}
\left\{\sum_{i=1}^k(\Lambda_{k+1}-\Lambda_i)
\left(\Lambda_i-\frac{n-2}{\Lambda_i-(n-2)}\right)\left(\Lambda_i+\frac{(n-2)^2}{
4}\right)\right\},
\endaligned$$
which together with inequality \eqref{2sec15} yields inequality
\eqref{Intro5}.

Solving the quadratic polynomial of $\Lambda_{k+1}$ in
\eqref{Intro5}, we obtain inequality \eqref{Intro6}.

Replacing $k+1$ with $k$ in \eqref{Intro5}, we obtain
$$\sum_{i=1}^{k-1}(\Lambda_k-\Lambda_i)^2\leq
\sum_{i=1}^{k-1}(\Lambda_k-\Lambda_i)\left(\Lambda_i-\frac{n-2}{\Lambda_i-(n-2)}\right)
\left(\Lambda_i+\frac{(n-2)^2}{4}\right).$$ Therefore,
$$\sum_{i=1}^k(\Lambda_k-\Lambda_i)^2\leq
\sum_{i=1}^k(\Lambda_k-\Lambda_i)\left(\Lambda_i-\frac{n-2}{\Lambda_i-(n-2)}\right)
\left(\Lambda_i+\frac{(n-2)^2}{4}\right).$$ Solving this inequality
of quadratic polynomial for $\Lambda_{k}$, we infer
\begin{equation}\label{2sec16}\Lambda_k\geq
S_{k+1}-\sqrt{S_{k+1}^2-T_{k+1}},\end{equation} where $S_{k+1},
T_{k+1}$ are given by \eqref{Intro8} and \eqref{Intro9}
respectively. Therefore, the inequality \eqref{Intro7} follows from
\eqref{2sec16} and \eqref{Intro6}. Then Corollary 1.3 is completed.

\end{document}